\documentclass[10pt]{amsart}
\usepackage{amssymb,mathrsfs,graphicx,enumerate}
\usepackage{amsmath,amsfonts,amssymb,amscd,amsthm,bbm}
\usepackage{algorithm}
\usepackage{algpseudocode}
\usepackage[retainorgcmds]{IEEEtrantools}
\usepackage{colortbl}

\title[Optimal consensus control model]{Optimal consensus control models on the sphere}

\author[Huang]{Hui Huang}
\address[Hui Huang]{\newline Institute of Mathematics and Scientific Computing,
University of Graz, Universit\"{a}tspl. 3, 8010 Graz, Austria}
\email{hui.huang@uni-graz.at}

\author[Park]{Hansol Park}
\address[Hansol Park]{\newline Department of Mathematics, Simon Fraser University, 8888 University Dr, Burnaby, BC V5A 1S6, Canada}
\email{hansol$\_$park@sfu.ca}
\email{hansol960612@snu.ac.kr}

\newtheorem{theorem}{Theorem}[section]

\newcommand{\bbr}{\mathbb R}
\newcommand{\R}{\mathbb R}

\newcommand{\bbs}{\mathbb S}

\newcommand{\bx}{\mbox{\boldmath $x$}}

\newcommand{\bv}{\mbox{\boldmath $v$}}

\newcommand{\dd}{\mathrm{d}}

\newcommand{\la}{\langle}
\newcommand{\ra}{\rangle}

\begin{document}

\date{\today}


\thanks{\textbf{Acknowledgment.} The work of H. Park is supported by Pacific Institute for the Mathematical Science (PIMS), Canada postdoctoral fellowship.}

\begin{abstract}
In this paper, we investigate the consensus models on the sphere with control signals, where both the first and second order systems are considered. We provide the existence of the optimal control-trajectory pair and derive the first order optimality condition taking the form of the Pontryagin Minimum Principle. Numeric simulations are also presented to show that the obtained optimal control can help to accelerate the process of reaching a consensus.
\end{abstract}

\maketitle \centerline{\date}


\section{Introduction}
Large systems of interacting particles arise in the study of collective behaviours of biological and physical systems on manifolds, such as nematic patterns on a sphere \cite{sanchez2012spontaneous}, application to launching unmanned spacecrafts \cite{chang2010cooperative}, and crystals on a cylinder \cite{doi1981molecular}. It is receiving lots of attention due to the appearance of a consensus emergence of these particle systems.
This paper addresses centralized control problems for one of the well known consensus models on the sphere, i.e. the swarm sphere model \cite{Lo}, whose first order form is given by:
{\small \begin{align}\label{a-1}
\begin{cases}
\displaystyle \frac{\dd}{\dd t}x_i=\Omega_i x_i+\frac{\kappa}{N}\sum_{k=1}^N\left(x_k-\frac{\langle x_i, x_k\rangle}{\|x_i\|^2} x_i\right),\\
x_i(0)=x_i^0\in\bbs^{d-1},\quad \forall ~i\in[N]:=\{1, 2, \cdots, N\},
\end{cases}
\end{align}}
where $\Omega_i\in\mathrm{Skew}(d):=\{A\in\bbr^{d\times d}: A^\top=-A\}$ is the natural frequency of the $i$-th particle, $\kappa$ is the coupling strength, $x_i^0$ is the initial position of the $i$-th particle, $N$ is the number of particles, and $\|\cdot\|$ represents the standard Euclidean norm. Adding the inertia (mass) term to system \eqref{a-1} as in \cite{H-K} we can obtain the second order model satisfying:
{\small \begin{align}\label{a-2}
\begin{cases}
\displaystyle \frac{\dd}{\dd t}x_i=v_i,\\
\displaystyle \frac{\dd}{\dd t}v_i=-\frac{\gamma}{m} v_i-\frac{\|v_i\|^2}{\|x_i\|^2} x_i+\frac{1}{m}\Omega_i x_i+\frac{\kappa}{mN}\sum_{k=1}^N\left(x_k-\frac{\langle x_i, x_k\rangle}{\|x_i\|^2} x_i\right),\\
\displaystyle x_i(0)=x_i^0\in\bbs^{d-1},\quad v_i(0)= \frac{\dd}{\dd t}x_i\Big|_{t=0+}=v_i^0\in T_{x_i^0}\bbs^{d-1}\quad \forall ~i\in[N],
\end{cases}
\end{align}}
where $m$ is the mass, $\gamma$ is the friction coefficients, and $v_i^0$ is the initial velocity of the $i$-th particle. Actually, the systems \eqref{a-1} and \eqref{a-2} can be also seen as generalized Kuramoto models on the sphere. 
According to \cite{H-K}, if the identical natural frequencies satisfies $\Omega_i\equiv O_{d-1}$ for all $i\in [N]$, then the emergent dynamics for both systems \eqref{a-1} and \eqref{a-2} with generic initial can be observed, namely it satisfies
{\small \[
\lim_{t\to\infty}\|x_i(t)-x_j(t)\|=0\quad\forall~i, j\in [N].
\]}
In this context the consensus is understood as a travelling formulation in which every particle has the same position. However this behaviour strongly depends on the initial configuration of the particle system. In the present work we are interested in the design of centralized control signals enforcing consensus emergence in systems \eqref{a-1} and \eqref{a-2} in order to generate an external intervention able to steer the system towards a desired configuration. Very related to this work, \cite{BBCK,caponigro2013sparse,bongini2015conditional} study the problem of consensus control for Cucker--Smale type models. 
Meanwhile, according to \cite{HLRS}, the Kuramoto model (position synchronization model) can be derived from the Cucker--Smale model (velocity alignment model). So in this work we are aiming to extend those results on velocity alignment as in \cite{BBCK} to our position synchronization models \eqref{a-1} and \eqref{a-2}.


 

\section{Swarm sphere models with controls}\label{sec:2}
In this section we consider both the first order and second order swarm sphere models with control signals. 
\subsection{Second order model}
We consider a set of $N$ particles with state $(x_i(t),v_i(t))\in T\bbs^{d-1}\subset \R^d\times\R^d$ moving on a sphere via system \eqref{a-2} with $\Omega_i\equiv O_{d-1}$ for all $i\in [N]$.
We are interested in the study of consensus emergence, i.e. the convergence towards a configuration in which $
	x_i=\bar x=\frac{1}{N}\sum_{j=1}^Nx_j,~\forall\, i\in[N]\,.$
In particular we are concerned with inducing the consensus through the synthesis of an external forcing term $\mathbf{u}(t):=(u_1(t),\dots,u_N(t))^\top$ in the form of
{\small \begin{align}\label{Kucontro}
	\begin{cases}
		\displaystyle\frac{\dd x_i}{dt}=v_i,\\
		\displaystyle\frac{\dd v_i}{dt}=-\frac{\|v_i\|^2}{\|x_i\|^2}x_i-\frac{\gamma}{m}v_i+\frac{\kappa}{mN}\sum_{j=1}^N\left(x_j-\frac{\la x_i,x_j\ra}{\|x_i\|^2}x_i\right)+\left(u_i-\frac{\la u_i,x_i\ra x_i}{\|x_i\|^2}\right)\\
\displaystyle x_i(0)=x_i^0\in\bbs^{d-1}, \quad v_i(0)=\frac{\dd}{\dd t}x_i\Big|_{t=0+}=v_i^0\in T_{x_i^0}\bbs^{d-1}\quad \forall ~i\in[N],
	\end{cases}
\end{align}}
where the control signals $u_i\in L^2([0,T];\R^d)=:\mathcal{U}$. Formally, for $T>0$ and given a set of admissible control signals $\textbf{u}\in \mathcal{U}^N$ for the entire population, it holds that 
$\|u_i\|_2\leq \max_{i\in[N]}\{\|u_i\|_2\}=:M$ for all $i\in[N]$. Our goal is then to seek a solution to the minimization problem
{\small \begin{align}\label{min}
	\min\limits_{\mathbf{u}(\cdot)\in \mathcal{U}^N} \mathcal{J}(\mathbf{u}(\cdot);\mathbf{x}(0),\mathbf{v}(0)):=\int_0^T\frac{1}{N}\sum_{j=1}^{N}\|x_j-\bar x\|^2dt+\lambda\int_0^T\frac{1}{N}\sum_{j=1}^{N}\|u_j\|^2dt\,,
\end{align}}
with some regularization parameter $\lambda>0$. Here we have used the notation $\mathbf{x}(t):=(x_1(t),\dots,x_N(t))^\top$ and $\mathbf{v}(t):=(v_1(t),\dots,v_N(t))^\top$.
\begin{theorem}\label{thmexistence}
	For any given $T>0$ and $\mathbf{u}(\cdot)\in \mathcal{U}^N$, assume the initial data $(\mathbf{x}(0),\mathbf{v}(0))$ and parameters $m,\gamma,\kappa,T$ satisfy
{\small 	$$\frac{m}{\gamma}\left(\mathcal{V}(0)+\frac{2\kappa T}{m}+2MT^{\frac{1}{2}}\right)\left(\exp\left(\frac{\gamma T}{m}\right)-1\right)<1\,,$$}
	then there exists a pathwise unique global solution $\{(x_i,v_i)\}_{i=1}^N$ to the Kuramoto system \eqref{Kucontro} up to time $T$. Moreover for all $i\in[N]$ and $t\in[0,T]$ it holds that 
{\small 	\begin{equation}\label{Xbd}
	\|x_i(t)\|=1, \quad \la x_i(t),v_i(t)\ra=0
	\end{equation}}
	and
{\small 	\begin{equation}\label{Vbd}
			\sup\limits_{t\in[0,T]}\mathcal{V}(t)\leq \frac{\exp\left(\frac{\gamma T}{m}\right)\left(\mathcal{V}(0)+\frac{2\kappa T}{m}+2MT^{\frac{1}{2}}\right)}{1-\frac{m}{\gamma}\left(\mathcal{V}(0)+\frac{2\kappa T}{m}+2MT^{\frac{1}{2}}\right)\left(\exp\left(\frac{\gamma T}{m}\right)-1\right)}=:C_V\,,
	\end{equation}}
where $\mathcal{V}(t)=\max\limits_{i=1,\dots,N}\|v_i(t)\|$.
\end{theorem}
\begin{proof}
	Let $\phi$ be a map on $\R^d$ with bounded derivatives of all orders such that $\phi(x)=\frac{x}{\|x\|^2}$ for all $x$ with $\|x\|\geq \frac{1}{2}$. Then we consider the following regularized system
{\small 	\begin{align}\label{RKucontro}
		\begin{cases}
			\displaystyle\frac{\dd x_i}{\dd t}=v_i,\\
			\displaystyle\frac{\dd v_i}{\dd t}=-\phi(x_i)\|v_i\|^2-\frac{\gamma}{m}v_i+\frac{\kappa}{m}\frac{1}{N}\sum_{j=1}^N(x_j-\la x_i,x_j\ra \phi(x_i))+u_i-\la u_i,x_i\ra \phi(x_i)\\
\displaystyle x_i(0)=x_i^0\in\bbs^{d-1}, \quad v_i(0)=\frac{\dd}{\dd t}x_i\Big|_{t=0+}=v_i^0\in T_{x_i^0}\bbs^{d-1}\quad \forall ~i\in[N],
		\end{cases}
	\end{align}}
It is obvious that  the coefficients are local Lipschitz, so one has local existence and uniqueness up to some time $\tau\in [0,T]$. 

Next we prove the global existence up to time $T$. Actually, as long as $\|x_i\|\geq \frac{1}{2}$ one has
{\small \begin{align*}
	\frac{\dd}{\dd t}\la x_i, v_i \ra=\|v_i\|^2+\left\la x_i,\frac{\dd v_i}{\dd t}\right\ra=\|v_i\|^2-\|v_i\|^2-\frac{\gamma}{m}\la x_i,v_i\ra=-\frac{\gamma}{m}\la x_i,v_i\ra\,,
\end{align*}}
which implies $\la x_i(t), v_i(t) \ra=\la x_i(0), v_i(0)\ra e^{-\frac{\gamma}{m}t}=0$ for all $t\in[0,T]$ since $\la x_i(0), v_i(0)\ra =0$. Then we have
{\small \begin{align*}
	\frac{\dd}{\dd t}\|x_i\|^2=2\la x_i, v_i \ra=0\,,
\end{align*}}
which following the initial condition $\|x_i(0)\|=1$ leads to $\|x_i(t)\|=1$ for all $t\in[0,T]$. In addition, notice that
{\small \begin{align*}
	\|v_i(t)\|\leq \|v_i(0)\|+\int_0^t(\|v_i(s)\|^2+\frac{\gamma}{m}\|v_i(s)\|)ds+\frac{2\kappa T}{m}+2\int_0^t\|u_i(s)\|ds\,.
\end{align*}}
Using Gronwall's lemma one has
{\small \begin{align*}
	\sup\limits_{t\in[0,T]}\|v_i(t)\|&\leq \frac{\exp(\frac{\gamma T}{m})(\|v_i(0)\|+\frac{2\kappa T}{m}+2\int_0^T\|u^i(s)\|ds)}{1-\frac{m}{\gamma}(\|v_i(0)\|+\frac{2\kappa T}{m}+2\int_0^T\|u^i(s)\|ds)(\exp(\frac{\gamma T}{m})-1)}\leq \frac{\exp(\frac{\gamma T}{m})(\|v_i(0)\|+\frac{2\kappa T}{m}+2MT^{\frac{1}{2}})}{1-\frac{m}{\gamma}(\|v_i(0)\|+\frac{2\kappa T}{m}+2MT^{\frac{1}{2}})(\exp(\frac{\gamma T}{m})-1)}
	\,.
\end{align*}}
This means that if initially $\frac{m}{\gamma}(\|v_i(0)\|+\frac{2\kappa T}{m}+2MT^{\frac{1}{2}})(\exp(\frac{\gamma T}{m})-1)<1$, then we have global existence and pathwise uniqueness for \eqref{RKucontro} up to time $T$. 

Now the solution to \eqref{RKucontro} for the given $\phi$ is a solution to \eqref{Kucontro} since $\|x_i\|=1$ for all $i=1,\dots,N$ and $t\in[0,T]$, which provides global existence to \eqref{Kucontro}. If we consider two solutions to \eqref{Kucontro} for the same initial data, then they satisfies $\|x_i\|=1$. So they are also solutions to the regularized system \eqref{RKucontro}, for which pathwise uniqueness holds. Hence they are equal, which completes the proof.
\end{proof}
From \eqref{Xbd}, now we can assume $\|x_i(t)\|=1$ for all $t\geq0$ and $i\in [N]$ if $\{(x_i, v_i)\}_{i=1}^N$ satisfy \eqref{Kucontro},  and one can prove the existence of the optimal control-trajectory pair:

\begin{theorem}\label{thmcon}
	Under the same assumptions as in Theorem \ref{thmexistence}, there exists some control $u_i^*\in L^2(0,T;\R^d)$, $i\in[N]$ and the corresponding $\{(x_i^*,v_i^*)\}_{i=1}^N$ trajectories solving the optimal control problem \eqref{Kucontro}-\eqref{min}.
\end{theorem}
\begin{proof}
	For any given $u_i\in L^2(0,T;\R^d)$, $i\in[N]$, by Theorem \ref{thmexistence} there exists a corresponding solution  $\{(x_i^u,v_i^u)\}_{i=1}^N$ to \eqref{Kucontro}. Note that $u_i=0\in L^\infty(0,T;\R^d)$, so that
{\small 	\begin{equation*}
		\mathcal{J}(\mathbf{0};\mathbf{x}(0),\mathbf{v}(0))=\int_0^T\frac{1}{N}\sum_{j=1}^{N}\|x_j^0-\bar x^0\|^2\dd t\leq 4T
	\end{equation*}}
because $\|x_i^0\|=1$, where $\{(x_i^0,v_i^0)\}_{i=1}^N$ is the solution to \eqref{Kucontro} with given $u_i=0$. Since $\mathcal{J}(\mathbf{u}(\cdot);\mathbf{x}(0),\mathbf{v}(0))\geq 0$ for all $\mathbf{u}(\cdot)\in \mathcal{U}^N$, it holds that
{\small \begin{equation*}
0\leq \min\limits_{\mathbf{u}(\cdot)\in \mathcal{U}^N}\mathcal{J}(\mathbf{u}(\cdot);\mathbf{x}(0),\mathbf{v}(0))\leq \mathcal{J}(\mathbf{0};\mathbf{x}(0),\mathbf{v}(0))\leq 4T\,.	
\end{equation*}}
By the definition of the infimum, we know that there exists a sequence  of controls $(\mathbf{u}^n)_{n\in \mathbb N}\subset \mathcal{U}^N$ with corresponding $(\bx^{u^n},\bv^{u^n})$ solving \eqref{Kucontro}, such that 
{\small \begin{equation*}
	0\leq \min\limits_{\mathbf{u}(\cdot)\in \mathcal{U}^N}\mathcal{J}(\mathbf{u}(\cdot);\mathbf{x}(0),\mathbf{v}(0))\leq \mathcal{J}(\mathbf{u}^n(\cdot);\mathbf{x}(0),\mathbf{v}(0))\leq \min\limits_{\mathbf{u}(\cdot)\in \mathcal{U}^N}\mathcal{J}(\mathbf{u}(\cdot);\mathbf{x}(0),\mathbf{v}(0))+\frac{1}{n}\,,
\end{equation*}}
and so
{\small \begin{equation*}
	\lim_{n\to\infty}\mathcal{J}(\mathbf{u}^n(\cdot);\mathbf{x}(0),\mathbf{v}(0))= \min\limits_{\mathbf{u}(\cdot)\in \mathcal{U}^N}\mathcal{J}(\mathbf{u}(\cdot);\mathbf{x}(0),\mathbf{v}(0))\,.
\end{equation*}}
Notice that
{\small \begin{equation*}
	\lambda\int_0^T\frac{1}{N}\sum_{j=1}^{N}\|u_j^n\|^2\dd t\leq \mathcal{J}(\mathbf{u}^n(\cdot);\mathbf{x}(0),\mathbf{v}(0))<\infty\,,
\end{equation*}}
and by Banach-Alaoglu theorem, for all $i$ there exists a subsequence $(u_i^{n_k})_{k\in \mathbb N}$ and $u_i^*\in L^2(0,T;\R^d)$ such that
{\small \begin{equation}\label{weakcon}
u_i^k:=u_i^{n_{k}} \underset{k \rightarrow+\infty}{\rightharpoonup} u_i^* \quad \text { in } L^{2}(0, T; \mathbb{R}^{d})\,.
\end{equation}}
For the corresponding solutions $(\bx^{k},\bv^{k})_{k\in\mathbb N}:=(\bx^{u^{n_k}},\bv^{u^{n_k}})_{k\in\mathbb N}$, we have from Theorem \ref{thmexistence} that
{\small \begin{align*}
	\max\limits_{i=1,\dots,N}\|\dot{v_i}^k(t)\|&\leq \max\limits_{i=1,\dots,N}\|v_i^k(t)\|^2+\frac{\gamma}{m}\max\limits_{i=1,\dots,N}\|v_i^k(t)\|+\frac{2\kappa}{m}+2\max\limits_{i=1,\dots,N}\|u_i^k(t)\|\leq C_V^2+\frac{C_V\gamma}{m}+\frac{2\kappa}{m}+2M\,.
\end{align*}}
This combining with \eqref{Vbd} implies the equi-boundedness and equi-absolute continuity of $v_i^k(t)$ uniformly with respect to $k$, for all $i=1,\dots,N$. This also yields the equi-Lispchitzanity of $x_i^k(t)$. By Ascoli-Arzela theorem there exits a subsequence, again renamed  $(\bx^{k},\bv^{k})_{k\in\mathbb N}$ and an absolutely continuous trajectory $(\bx^*,\bv^*)$ in $[0,T]$ such that for $k\to\infty$:
{\small \begin{align}\label{stongcon}
\begin{cases}
x_i^k \rightarrow x_i^*,\mbox{ in }[0,T],\mbox{ for all }i=1,\cdots,N,\\
v_i^k \rightarrow v_i^*,\mbox{ in }[0,T],\mbox{ for all }i=1,\cdots,N,\\
\dot{x}_i^k \rightarrow \dot{x}_i^*,\mbox{ in }[0,T],\mbox{ for all }i=1,\cdots,N,\\
\dot{v}_i^k \rightharpoonup\dot{v}_i^*,\mbox{ in }L^2(0,T;\R^d),\mbox{ for all }i=1,\cdots,N\,.
\end{cases}
\end{align}}
Thus  it is easy to see that
{\small \begin{equation}\label{eqstar1}
\frac{\dd x_i^*}{\dd t}=v_i^* \mbox{ and }
\lim\limits_{k\to \infty}\int_0^T\frac{1}{N}\sum_{j=1}^{N}\|x_j^k-\bar x^k\|^2\dd t=\int_0^T\frac{1}{N}\sum_{j=1}^{N}\|x_j^*-\bar x^*\|^2\dd t\,.
\end{equation}}
Moreover, one notice that for all $\psi\in L^2(0,T;\R^d)$ 
{\small \begin{align*}
&\int_0^T\psi\dot{v}_i^k \dd t=\int_0^T\left(-\frac{\|v_i^k \|^2}{\|x_i^k \|^2}\la x_i^k ,\psi\ra-\frac{\gamma}{m}\la v_i^k ,\psi\ra+\frac{\kappa}{m}\frac{1}{N}\sum_{j=1}^N\left(\la x_j^k ,\psi\ra-\frac{\la x_i^k ,x_j^k \ra}{\|x_i^k \|^2}\la x_i^k ,\psi\ra\right)+\la u_i^k ,\psi\ra-\frac{\la u_i^k ,x_i^k \ra \la x_i^k,\psi\ra}{\|x_i^k \|^2}\right) \dd t
\end{align*}}
Let $k\to \infty$, and by using \eqref{weakcon} and \eqref{stongcon} we have
{\small \begin{align*}
&\int_0^T\psi\dot{v}_i^*\dd t=\int_0^T\left(-\frac{\|v_i^*\|^2}{\|x_i^* \|^2}\la x_i^* ,\psi\ra-\frac{\gamma}{m}\la v_i^* ,\psi\ra+\frac{\kappa}{m}\frac{1}{N}\sum_{j=1}^N\left(\la x_j^*,\psi\ra-\frac{\la x_i^*,x_j^* \ra}{|x_i^* |^2}\la x_i^* ,\psi\ra\right)+\la u_i^* ,\psi\ra-\frac{\la u_i^*,x_i^* \ra \la x_i^*,\psi\ra}{\|x_i^* \|^2}\right) \dd t\,,
\end{align*}}
which leads to
{\small \begin{equation}\label{eqstar2}
\frac{\dd v_i^*}{\dd t}=-\frac{\|v_i^*\|^2}{\|x_i^*\|^2}x_i^*-\frac{\gamma}{m}v_i^*+\frac{\kappa}{m}\frac{1}{N}\sum_{j=1}^N\left(x_j^*-\frac{\la x_i^*,x_j^*\ra}{\|x_i^*\|^2}x_i^*\right)+u_i^*-\frac{\la u_i^*,x_i^*\ra x_i^*}{\|x_i^*\|^2}\quad \mbox{a.e.}\,.
\end{equation}}
Since $(u_i^k)_{k\in\mathbb N}\subset L^2(0,T;\R^d)$ converges weakly to $u_i^*$ for all $i\in[N]$, we also have
{\small \begin{equation*}
\lambda\int_0^T\frac{1}{N}\sum_{j=1}^{N}\|u_j^*\|^2\dd t\leq \lim\inf\limits_{k\to \infty }\lambda\int_0^T\frac{1}{N}\sum_{j=1}^{N}\|u_j^k\|^2\dd t
\end{equation*}}
by the weak lower-semicontinuity of the $L^2$-norm. This implies that
{\small \begin{align*}
	&\mathcal{J}(\mathbf{u}^*(\cdot);\mathbf{x}(0),\mathbf{v}(0))=\int_0^T\frac{1}{N}\sum_{j=1}^{N}\|x_j^*-\bar x^*\|^2\dd t+\lambda\int_0^T\frac{1}{N}\sum_{j=1}^{N}\|u_j^*\|^2\dd t\notag\\
	\leq&\lim\inf\limits_{k\to \infty }\left(\int_0^T\frac{1}{N}\sum_{j=1}^{N}\|x_j^k-\bar x^k\|^2\dd t+\lambda\int_0^T\frac{1}{N}\sum_{j=1}^{N}\|u_j^k\|^2\dd t\right)\notag\\
	=&\lim\inf\limits_{k\to \infty }\mathcal{J}(\mathbf{u}^k(\cdot);\mathbf{x}(0),\mathbf{v}(0))=\lim_{n\to\infty}\mathcal{J}(\mathbf{u}^n(\cdot);\mathbf{x}(0),\mathbf{v}(0))= \min\limits_{\mathbf{u}(\cdot)\in \mathcal{U}^N}\mathcal{J}(\mathbf{u}(\cdot);\mathbf{x}(0),\mathbf{v}(0))\,.
\end{align*}}
This together with \eqref{eqstar1} and \eqref{eqstar2} implies the limit $(\textbf{x}^*,\textbf{v}^*,\textbf{u}^*)$ is a solution to the optimal control problem \eqref{Kucontro}-\eqref{min}.
\end{proof}

While the existence of $(\mathbf{x}^*,\mathbf{v}^*,\mathbf{u}^*)$ to the optimal control problem \eqref{Kucontro}-\eqref{min} has been obtained in Theorem \ref{thmcon}, the Pontryagin Minimum Principle \cite{Evans} yields first-order necessary conditions for the optimal control. Let $(p_i^*(t),q_i^*(t))\in \R^d\times \R^d$ be adjoint variables associated to $(x_i^*,v_i^*)$, and we set $\mathbf{p}^*=\{p_i^*\}_{i=1}^N$, $\mathbf{q}^*=\{q_i^*\}_{i=1}^N$. Then the optimality system consists of a solution $(\mathbf{x}^*,\mathbf{v}^*,\mathbf{u}^*,\mathbf{p}^*,\mathbf{q}^*)$ satisfying \eqref{Kucontro} along with the adjoint equations
{\small \begin{align}\label{secondadj}
	\begin{cases}
\displaystyle-\frac{\dd p_i^*}{\dd t}=-\|v_i^*\|^2q_i^*+\frac{\kappa}{mN}\sum_{j=1}^{N}\left(q_j^*-\la x_i^*,x_j^*\ra q_i^*-\la x_i^*,q_i^*\ra x_j^*-\la x_j^*,q_j^*\ra x_j^*\right)-q_i^*\la u_i^*,x_i^*\ra-u_i^*\la x_i^*,q_i^*\ra-\frac{2}{N}(\langle \bar{x}^*, \bar{x}_i^*\rangle\bar x^*-x_i^*)\,,\\
\displaystyle-\frac{\dd q_i^*}{\dd t}= p_i^*-2\la x_i^*,q_i^*\ra v_i^*-\frac{\gamma}{m}q_i^*\,,\\
p_i^*(T)=0,q_i^*(T)=0,\quad i\in[N]\,,
	\end{cases}
\end{align}}
and the optimality condition
{\small \begin{equation*}
	\mathbf{u}^*=\arg\min_{\mathbf{w}\in\R^{dN}}\sum_{j=1}^{N}\left(\left\la q_j^*,\frac{\dd v_j^*}{\dd t}\right\ra +\frac{\lambda}{N}|w_j|^2\ \right)=-\frac{N}{2\lambda}\left[q_j^*-\la q_j^*, x_j^*\ra x_j^*\right]_{j=1}^N.
\end{equation*}}

Recall that the gradient of the functional $\mathcal{J}$ introduces in \eqref{min} is given as follows:
{\small \begin{align}\label{secondgrad}
\nabla \mathcal{J}=\frac{2\lambda}{N}\mathbf{u}+\left[q_j-\la q_j, x_j\ra x_j\right]_{j=1}^N=\left[
\frac{2\lambda}{N}u_j+(q_j-\la q_j, x_j\ra x_j)
\right]_{j=1}^N.
\end{align}}
We will apply this gradient form to the gradient descent method in Section \ref{sec:3.2}.

\subsection{First order model}
We are also interested in the following first-order control problem:
{\small \begin{align}\label{C-1}
	\begin{cases}
		\displaystyle \frac{\dd}{\dd t}x_i=\frac{\kappa}{N}\sum_{k=1}^N\left(x_k-\langle x_i, x_k\rangle x_i\right)+u_i-\langle u_i, x_i\rangle x_i,\\
		x_i(0)=x_i^0\in\bbs^{d-1}\quad \forall ~i\in[N],
	\end{cases}
\end{align}}
with the following payoff functional
{\small \begin{equation}
	\tilde{\mathcal{J}}(\mathbf{u}(\cdot);\mathbf{x}(0))=\int_0^T\frac{1}{N}\sum_{i=1}^N\left(\|x_i-\bar{x}\|^2+\lambda\|u_i\|^2\right)\dd t.\label{C-1'}
\end{equation}}
The existence of the above optimal control problem can be obtained similarly as the second order model, so here we omit the proofs and only present the theorems:
\begin{theorem}\label{thmexistence1st}
	For any $T>0$ and given $\mathbf{u}(\cdot)\in \mathcal{U}^N$,
	there exists a pathwise unique global solution $\{x_i\}_{i=1}^N$ to the Kuramoto system \eqref{C-1} up to time $T$. Moreover for all $i\in[N]$ and $t\in[0,T]$ it holds that $\|x_i(t)\|=1$.
\end{theorem}
\begin{theorem}
	There exists some control $u_i^*\in L^2(0,T;\R^d)$, $i\in[N]$ and the corresponding $\{x_i^*\}_{i=1}^N$ trajectories solving the optimal control problem \eqref{C-1}-\eqref{C-1'}.
\end{theorem}

Let $\{p_i^*\}_{i=1}^N$ be adjoint variables associated to $\{x_i^*\}_{i=1}^N$. The corresponding PMP equation is given by
{\small \begin{align}\label{adjeqfirst}
	-\frac{\dd}{\dd t}p_i^*
=\frac{\kappa}{N}\sum_{j=1}^N(p_j^*-\la x_i^*, p_i^*\ra x_j^*-\la x_j^*, p_j^*\ra x_j^*-\la x_i^*, x_j^*\ra p_j^*)-\la u_i^*, x_i^*\ra p_i^*-\la x_i^*, p_i^*\ra u_i^*
	+\frac{2}{N}(x_i^*-\bar{x}^*\la x_i^*, \bar{x}^*\ra)\,,
\end{align}}
with the optimality condition
{\small \[
\textbf{u}^*=\arg\min_{\textbf{w}\in \R^{dN}}\sum_{i=1}^N\left(
\left\langle p_j^*, \frac{\dd x_j^*}{\dd t}\right\rangle+\frac{\lambda}{N}\|w_i\|^2
\right)=-\frac{N}{2\lambda}[p_j^*-\langle p_j^*, x_j^*\rangle x_j^*]_{j=1}^N.
\]}
Also, the gradient of functional $\tilde{\mathcal{J}}$ introduced in \eqref{C-1'} can be expressed as
{\small \begin{align}\label{gradfirst}
\nabla \tilde{\mathcal{J}}=\frac{2\lambda}{N}\mathbf{u}+[p_j+\langle p_j, x_j\rangle x_j]_{j=1}^N.
\end{align}}
We will apply this gradient form to the gradient descent method in Section \ref{sec:3.1}.

\section{Algorithms and numeric simulations}\label{sec:3}
In this section, we provide numeric simulations of controlled systems proposed in Section \ref{sec:2} in order to give a simple and immediate illustration of how the optimal control signal can be used to accelerate the process of reaching a consensus.

\subsection{First order model with control}\label{sec:3.1}
Firstly, we consider the first order control problem \eqref{C-1}-\eqref{C-1'}. To minimize the payoff functional $\tilde{\mathcal{J}}$ defined in \eqref{C-1'}, we apply the gradient descent with Barzilai--Borwein method \cite{BB} with the explicit form of $\nabla\tilde{\mathcal{J}}$ given in \eqref{gradfirst}.

{\small \begin{algorithm}
	\caption{An algorithm with caption}\label{FirstAlg}
	\begin{algorithmic}
		\Require $tol>0$, $k_{max}$, $\mathbf{u}^0$, $\mathbf{u}^{-1}$.
		\State $k=0$;
		\While{$\|\nabla\tilde{\mathcal{J}}\|>tol$ and $k<k_{max}$}
		\State 1) Obtain $\mathbf{x}^k$ from \eqref{C-1} with $\mathbf{u}^k$;
		\State 2) Obtain $\mathbf{p}^k$ from \eqref{adjeqfirst} with $\mathbf{x}^k$ and $\mathbf{u}^k$;
		\State 3) Compute the learning rate based on the Barzilai--Borwein method:
		\[
		\alpha_k:=\frac{\langle \mathbf{u}^k-\mathbf{u}^{k-1}, \nabla\tilde{\mathcal{J}}(\mathbf{u}^k)-\nabla\tilde{\mathcal{J}}(\mathbf{u}^{k-1})\rangle}{\|\nabla\tilde{\mathcal{J}}(\mathbf{u}^k)-\nabla\tilde{\mathcal{J}}(\mathbf{u}^{k-1})\|^2}\,;
		\]
		\State 4) Update $\mathbf{u}^{k+1}=\mathbf{u}^k-\alpha_k\nabla\tilde{\mathcal{J}}(\mathbf{u}^k)$;
		\State 5) $k:=k+1$;
		\EndWhile
	\end{algorithmic}
\end{algorithm}}
In Algorithm \ref{FirstAlg}, $k_{max}$ is the maximum of the number of iteration, $tol$ is the tolerance of $\|\nabla\tilde{\mathcal{J}}\|$, and we use the Runge-Kutta fourth order methods in Step 1) and Step 2).

In Figure \ref{fig1}, we compare the optimal controlled system \eqref{C-1}-\eqref{C-1'} and the control-free swarm sphere model \eqref{a-1} with the same initial data and parameters being set as
{\small \begin{align}\label{para}
N=20,\quad d=3,\quad \Delta t=0.01,\quad T=4,\quad \lambda=0.1,\quad \kappa=0.5.
\end{align}}
 As it is shown in Figure \ref{fig1} (a), the controlled system (blue) under an approximated optimal control $\mathbf{u}^*$ found with Algorithm \ref{FirstAlg} approaches to a consensus state much faster than the control-free model (red).  
 This can also be seen in Figure \ref{fig1} (b) that while the particles of the controlled system reaches a consensus, particles of the control-free system just move small distances. 
 In summary, for the first order model, the controlled system reaches consensus much faster than the uncontrolled system.

\begin{figure}[h]
\begin{tabular}{cc}
\includegraphics[width=4cm]{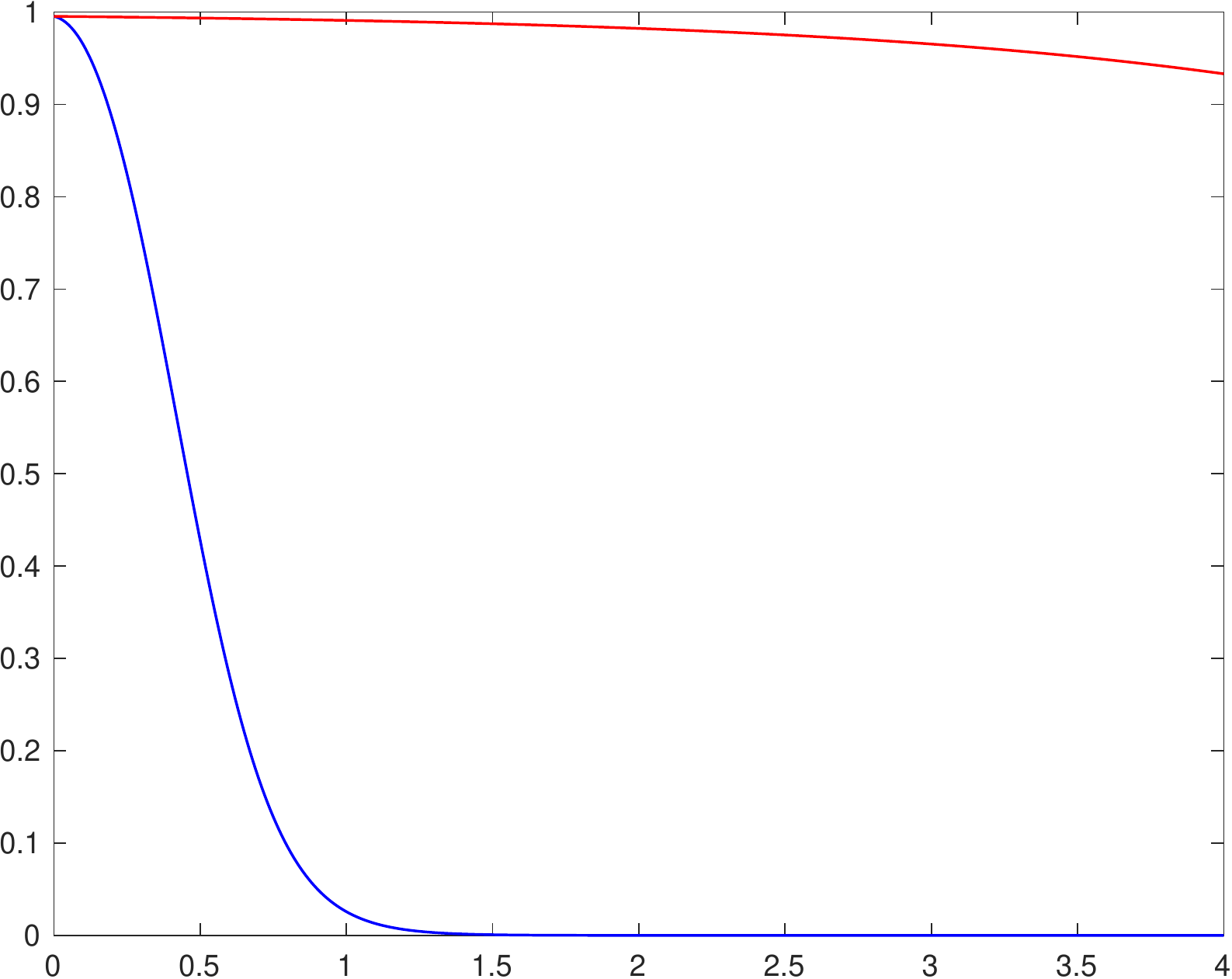}&
\includegraphics[width=5cm]{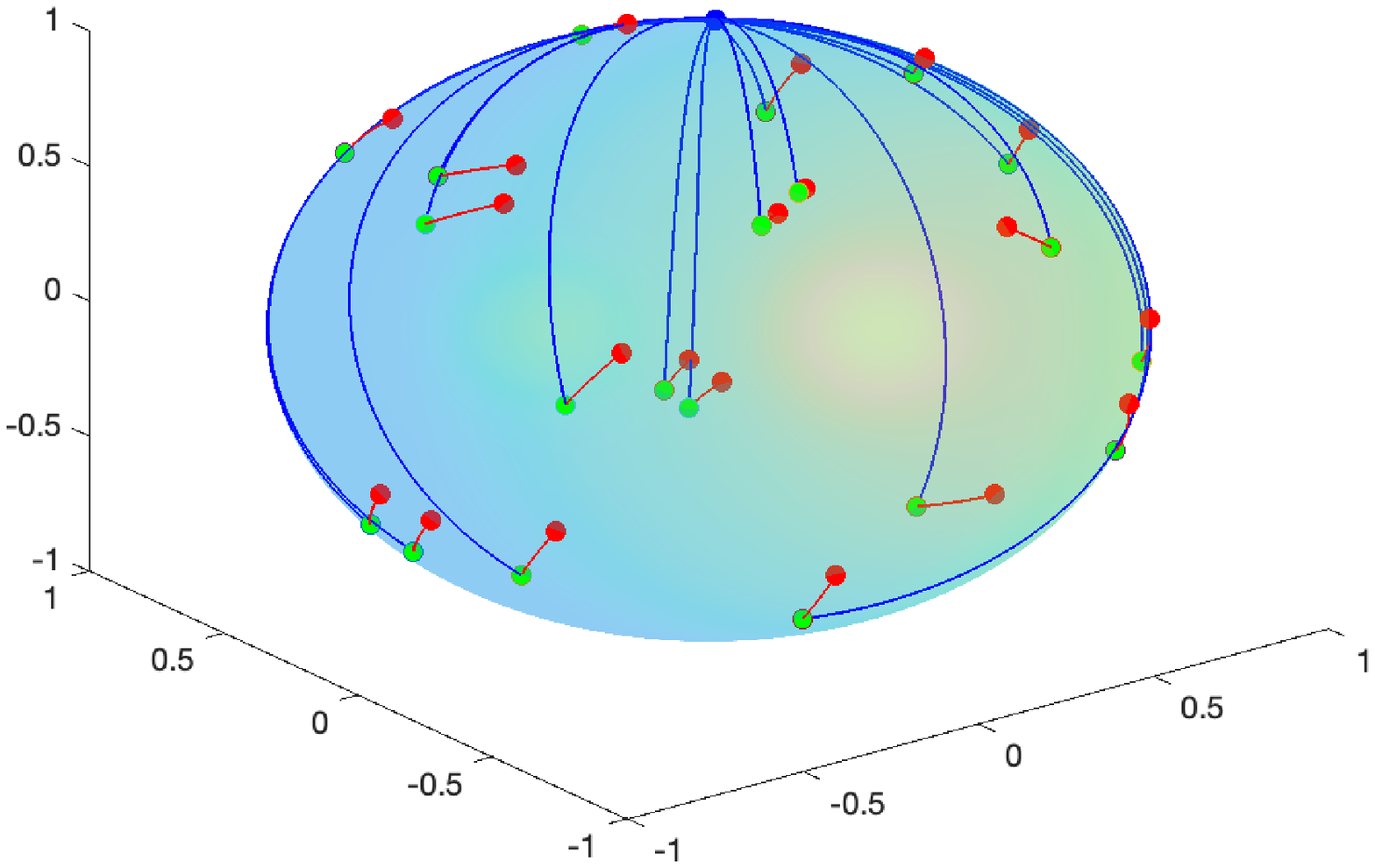}\\
(a) Variations of positions&
(b) Trajectories of particles
\end{tabular}

\caption{Controlled (blue) vs. control-free (red) first order dynamics. Figure (a) plots the time-evolution of the variations of positions defined as $\frac{1}{N}\sum_{i=1}^N\|x_i-\bar{x}\|^2$. Figure (b) shows the trajectories of two systems, where both of them start from the same green initial points. At the end time $T$, the optimal controlled particles reach the unique blue consensus point, while the uncontrolled particles move only small distances to red points (no consensus is obtained).
 }\label{fig1}
\end{figure}

\subsection{Second order model with control}\label{sec:3.2}
Now, we consider the second order control problem \eqref{Kucontro}-\eqref{min}. To minimize the payoff functional $\mathcal{J}$, we use the explicit form of $\nabla\mathcal{J}$ introduced in \eqref{secondgrad}. Same parameters given by \eqref{para} are used here with additionally setting the remained parameters as $m=1$ and $\gamma=1$. In  Figure \ref{fig2} we give a comparison of the second order optimal control dynamics \eqref{Kucontro}-\eqref{min} and its control-free counterpart \eqref{a-2}.

\begin{figure}[h]
\begin{tabular}{cccc}
\includegraphics[width=3.5cm]{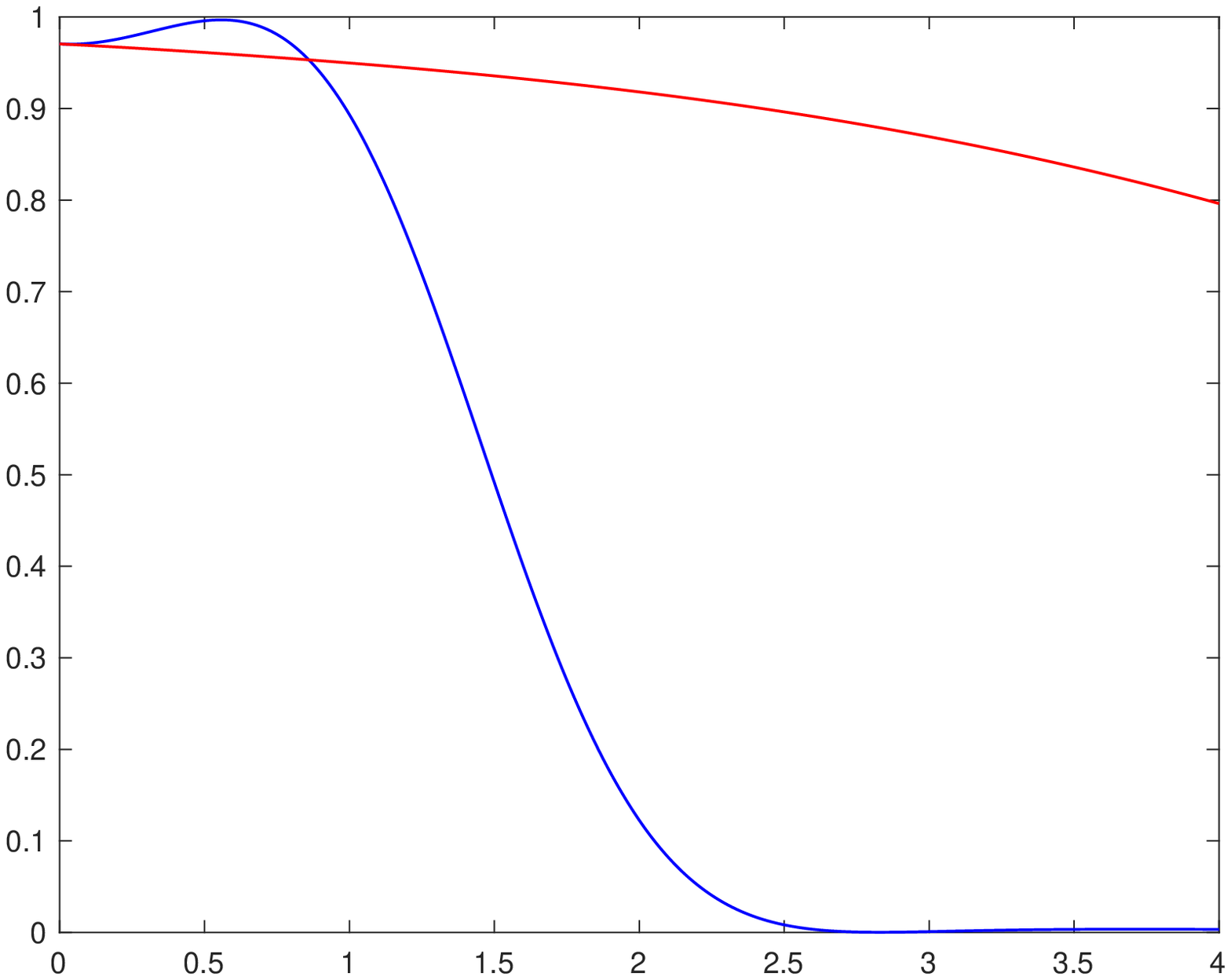}&
\includegraphics[width=3.5cm]{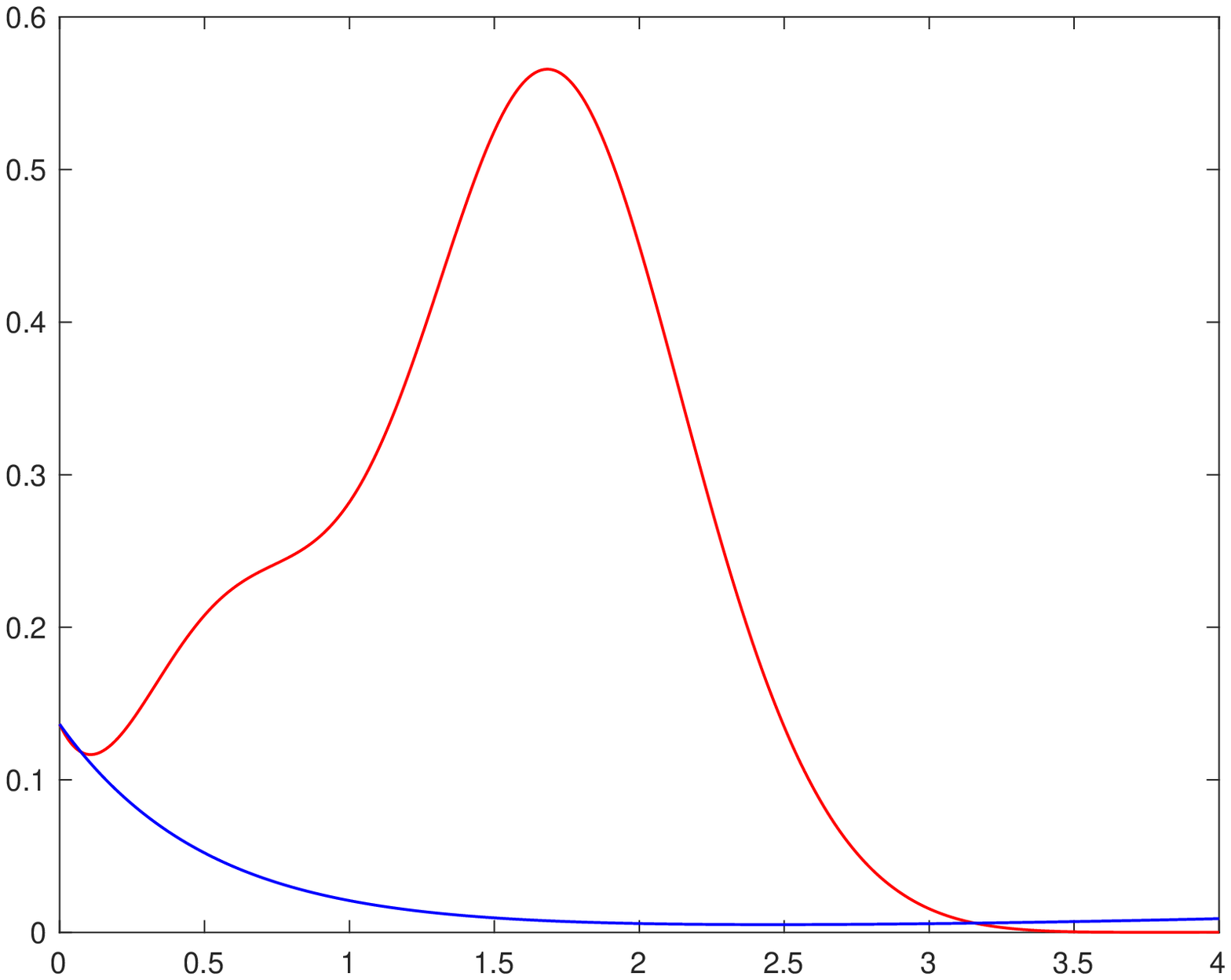}&
\includegraphics[width=4.3cm]{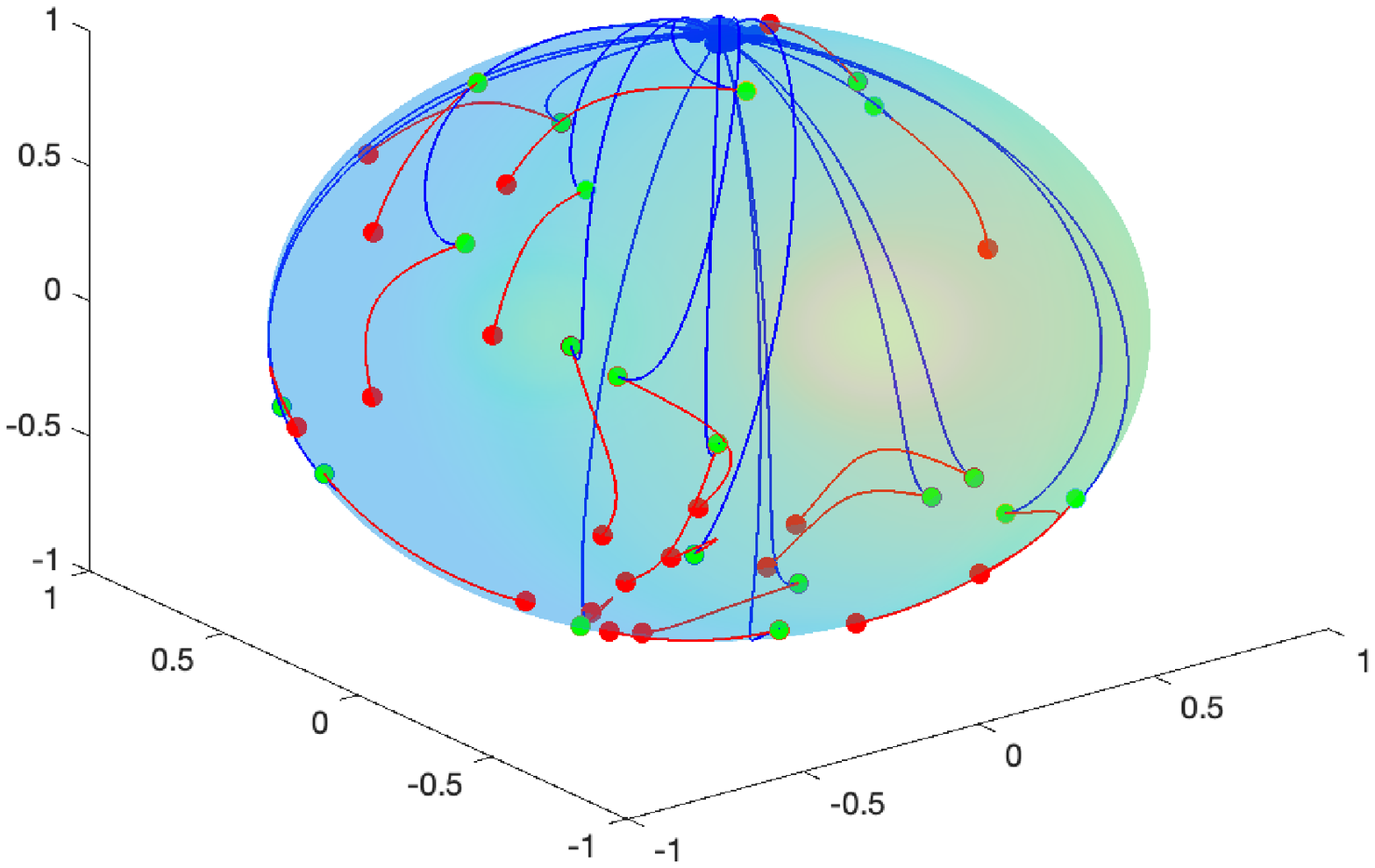}\\
(a) Variation of positions &
(b) Variations of velocities&
(c) Trajectories of particles
\end{tabular}
\caption{
Controlled (blue) vs. control-free (red) second order dynamics.
Figure (a) and (b) plot the time-evolution of the variations of positions/velocities, i.e. $\frac{1}{N}\sum_{i=1}^N\|x_i-\bar{x}\|^2$ and $\frac{1}{N}\sum_{i=1}^N\|v_i-\bar{v}\|^2$ respectively. Figure (c) shows the trajectories of two systems. }\label{fig2}
\end{figure}

\bibliographystyle{abbrv}
\bibliography{contro}

\end{document}